\documentclass[reqno]{amsart}

\input xy
\xyoption{all}
\usepackage{epsfig}
\usepackage{color}
\usepackage{amsthm}
\usepackage{amssymb}
\usepackage{amsmath}
\usepackage{amscd}
\usepackage{amsopn}
\usepackage{url}
\usepackage{hyperref}\hypersetup{colorlinks}

\usepackage{stmaryrd}

\usepackage[T1]{fontenc}

\usepackage{color} 

\definecolor{darkred}{rgb}{1,0,0} 
\definecolor{darkgreen}{rgb}{0,.6,0}
\definecolor{darkblue}{rgb}{0,0,1}

\hypersetup{colorlinks,
linkcolor=darkblue,
filecolor=darkgreen,
urlcolor=darkred,
citecolor=darkgreen}

\numberwithin{equation}{section}
\newtheorem {Theorem}{Theorem}
\numberwithin{Theorem}{section}

\newtheorem {Lemma}[Theorem]{Lemma}

\newtheorem {Proposition}[Theorem]{Proposition}

\theoremstyle{definition}

\theoremstyle{remark}
\newtheorem{Remark}[Theorem]{Remark}
\newtheorem{Example}[Theorem]{Example}

\def    \eps    {\epsilon}

\newcommand{\FF}{{\mathcal F}}

\newcommand{\CA}{{\mathcal A}}

\newcommand{\CS}{{\mathcal S}}
\newcommand{\tCS}{\tilde{\mathcal S}}

\newcommand{\gap}{\operatorname{gap}}

\newcommand{\const}{{\mathit const}}

\newcommand{\tpi}{\tilde{\pi}}

\newcommand{\tH}{\tilde{H}}

\newcommand{\tA}{\tilde{\mathcal A}}

\newcommand{\Pp}{{\mathcal P}}
\newcommand{\PP}{{\mathcal P}}

\def    \nat    {{\natural}}

\def    \R      {{\mathbb R}}

\def    \Z      {{\mathbb Z}}
\def    \N      {{\mathbb N}}
\def    \Q      {{\mathbb Q}}

\def    \T      {{\mathbb T}}
\def    \CP     {{\mathbb C}{\mathbb P}}

\def    \12     {{\frac{1}{2}}}

\def    \p      {\partial}

\def    \Sp     {\operatorname{Sp}}

\def    \HF     {\operatorname{HF}}
\def    \sign     {\operatorname{sign}}

\def    \tCF    {\widetilde{\operatorname{CF}}{}}
\def    \tHF    {\widetilde{\operatorname{HF}}{}}

\def    \H      {\operatorname{H}}
\def    \Tor    {\operatorname{Tor}}

\def    \CF     {\operatorname{CF}}

\def    \bh     {\bar{h}}
\def    \bx     {\bar{x}}
\def    \by     {\bar{y}}
\def    \bz     {\bar{z}}

\def  \MUCZ  {\operatorname{\mu_{\scriptscriptstyle{CZ}}}}

\newcommand{\ff}{{\mathfrak f}}
\newcommand{\fc}{{\mathfrak c}}
\newcommand{\fh}{{\mathfrak h}}

\def \NT {N_{\scriptscriptstyle{T}}} 

\def \NS {N_{\scriptscriptstyle{S}}}

\begin{document}


\setlength{\smallskipamount}{6pt}
\setlength{\medskipamount}{10pt}
\setlength{\bigskipamount}{16pt}





\title[Non-contractible Periodic Orbits]{Non-contractible Periodic
  Orbits in Hamiltonian Dynamics on Closed Symplectic Manifolds}

\author[Viktor Ginzburg]{Viktor L. Ginzburg}
\author[Ba\c sak G\"urel]{Ba\c sak Z. G\"urel}

\address{BG: Department of Mathematics, University of Central Florida,
  Orlando, FL 32816, USA} \email{basak.gurel@ucf.edu}

\address{VG: Department of Mathematics, UC Santa Cruz, Santa Cruz, CA
  95064, USA} \email{ginzburg@ucsc.edu}


\date{\today} 

\thanks{The work is partially supported by the NSF grants DMS-1414685
  (BG) and DMS-1308501 (VG)}


\begin{abstract} 
  We study Hamiltonian diffeomorphisms of closed symplectic manifolds
  with non-contractible periodic orbits. In a variety of settings, we
  show that the presence of one non-contractible periodic orbit of a
  Hamiltonian diffeomorphism of a closed toroidally monotone or
  toroidally negative monotone symplectic manifold implies the
  existence of infinitely many non-contractible periodic orbits in a
  specific collection of free homotopy classes. The main new
  ingredient in the proofs of these results is a filtration of Floer
  homology by the so-called augmented action. This action is
  independent of capping, and, under favorable conditions, the
  augmented action filtration for toroidally (negative) monotone
  manifolds can play the same role as the ordinary action filtration
  for atoroidal manifolds.
\end{abstract}

\maketitle

\tableofcontents

\section{Introduction}
\label{sec:intro}

In this paper, focusing on closed symplectic manifolds, we study
Hamiltonian diffeomorphisms with non-contractible periodic orbits.  We
show that in a variety of settings the presence of one
non-contractible one-periodic orbit of a Hamiltonian diffeomorphism of
a closed symplectic manifold guarantees the existence of infinitely
many non-contractible periodic orbits.

More specifically, we concentrate on closed symplectic manifolds
$(M^{2n}, \omega)$ which are toroidally monotone or toroidally
negative monotone and Hamiltonians with at least one non-contractible
one-periodic orbit $x$. Then we prove that under minor assumptions on
the free homotopy class $\ff$ of $x$ and under certain dynamical and
Floer theoretic conditions on $x$ and $M$, the Hamiltonian has
infinitely many simple periodic orbits in the collection of free homotopy
classes $\ff^\N:=\{\ff^k\mid k\in\N\}$; see Theorems
\ref{thm:homology-hyperbolic} and \ref{thm:homology-chi}. The
conditions on $\ff$ are automatically satisfied when the homology
class $[\ff]$ is non-zero in $\H_1(M;\Z)/\Tor$ or when $\ff\neq 1$ and
$\pi_1(M)$ is hyperbolic and torsion free. These theorems partially
extend the main result of \cite{Gu:nc} from atoroidal symplectic
manifolds to toroidally monotone or negative monotone manifolds.  (We
show in the next section that there are numerous manifolds and
Hamiltonians meeting the requirements of the theorems.)  The
phenomenon we consider here is $C^\infty$-generic. To be more precise,
we show in Theorem \ref{thm:generic} that the presence of one
non-contractible periodic orbit $x$ in a class $\ff$ such that
$1\not\in \ff^\N$ implies $C^\infty$-generically the existence of
infinitely many non-contractible periodic orbits in $\ff^\N$. Finally,
we also refine the results from \cite{Gu:nc} for atoroidal symplectic
manifolds; see Theorem \ref{thm:atoro}.

All these results are manifestations of the same underlying phenomenon
that the presence of a periodic orbit which is geometrically or
homologically unnecessary forces a Hamiltonian system to have
infinitely many periodic orbits. (On a closed symplectic manifold
non-contractible periodic orbits are clearly unnecessary. For example,
a $C^2$-small autonomous Hamiltonian has only constant one-periodic
orbits.) This phenomenon of ``forced existence'' of infinitely many
periodic orbits is very general and has been observed in a variety of
other settings. For instance, the celebrated theorem of Franks,
\cite{Fr1, Fr2}, asserting that a Hamiltonian diffeomorphism (or,
even, an area preserving homeomorphism) of $S^2$ with at least three
fixed points must have infinitely many periodic orbits is a
prototypical result along these lines; see also \cite{LeC} for further
refinements and \cite{CKRTZ, Ke:JMD} for a symplectic topological
proof. Another instance is a theorem from \cite{GG:hyperbolic} that
for a certain class of closed monotone symplectic manifolds including
$\CP^n$ any Hamiltonian diffeomorphism with a hyperbolic fixed point
must necessarily have infinitely many periodic orbits.  Yet, the
specific question of forced existence for non-contractible periodic
orbits considered here is largely unexplored except for \cite{Gu:nc}
focusing on symplectically atoroidal manifolds. We refer the reader
to, e.g., \cite{Gu:hq} for some other related results and to
\cite{GG:CC-survey} for a detailed discussion of the phenomenon and
further references.

The proofs of our main theorems rely on the machinery of Floer
homology for non-contractible periodic orbits. In the course of the
last two decades, this version of Floer homology has been studied and
used in a number of papers, but usually in a more topological context
and focusing on Hamiltonians on open manifolds such as twisted or
ordinary cotangent bundles; see, e.g., \cite{BPS, BuHa, GL, Lee, Ni,
  SW, We, Xu}. (The recent works \cite{Ba15,PS} are closer to the
setting considered in this paper which on the conceptual level can be
thought of as a continuation of \cite{Gu:nc}.)  The main difficulty in
applying the technique to closed manifolds is that the global Floer
homology for non-contractible orbits vanish and, moreover, already for
closed surfaces, a Hamiltonian may have no non-contractible periodic
orbits of any period. (As a consequence, the Floer complex can be
trivial for all iterations.) Thus to infer that a Hamiltonian has a
number of periodic orbits, e.g., infinitely many, an additional input
is required. In our case, this is one non-contractible periodic orbit
of the Hamiltonian, which serves as a seed generating infinitely many
offsprings. The new periodic orbits are detected by analyzing the
change in certain filtered Floer homology groups under the iteration
of the Hamiltonian and using the ``stability of the filtered
homology''.  This argument shares many common elements with the
reasoning in, e.g., \cite{GG:gap, Gu:nc, Gu:hq}. We feel that
it can also be cast in the framework of the barcode and persistent
homology theory for Floer homology (cf.\ \cite{PS,UZ}) and it would be
interesting to see if a systematic use of this theory would lead to
new results in this class of questions.

There are three key ingredients to the proofs in this paper.

The main new component is the observation, perhaps of independent
interest, that under favorable circumstances the Floer homology for a
toroidally monotone or toroidally negative monotone manifold is
filtered by the so-called augmented action. The augmented action
$\tA_H$ is the difference $\CA_H-(\lambda/2)\Delta_H$ between the
standard symplectic action $\CA_H$ and the (renormalized) mean index
$\Delta_H$ of an orbit, where $\lambda$ is the monotonicity constant;
cf.\ \cite[Sect.\ 1.4]{GG:gaps}.  The key feature of the augmented
action is that it is independent of capping and hence is assigned to
the orbit itself.  When the augmented action gap is sufficiently
large, the Floer differential does not increase the augmented action,
and the augmented action filtration is defined.  In this case, the
filtration behaves similarly to the ordinary action filtration in the
aspherical or atoroidal case and can be used in the same way. (One
essential difference, which is a source of several complications, is
that the augmented action filtration is not strict: in general, the
Floer differential can connect orbits with equal augmented action even
if the gap is large.)

The other two ingredients are the stability of the filtered homology
already mentioned above and the ball-crossing theorem \cite[Thm.\
3.1]{GG:hyperbolic} used in one of the proofs. This theorem gives an
iteration-independent lower bound on the energy of a Floer trajectory
asymptotic to a hyperbolic orbit.

The paper is organized as follows. In Section \ref{sec:results}, we
set our conventions and notation, introduce the necessary notions,
state the main results of the paper, and discuss in detail the classes
of manifolds and Hamiltonians these results apply to. In Section
\ref{sec:augfil}, we define the augmented action filtration and
establish its key properties. Finally, in Section \ref{sec:proofs}, we
prove the main theorems of the paper.

\medskip
\noindent{\bf Acknowledgements.} The authors are grateful to Alberto
Abbondandolo, Andr\'es Koropecki, Patrice Le Calvez, Denis Osin for
useful discussions and suggestions.

\section{Main results}
\label{sec:results}

\subsection{Conventions and notation}
\label{sec:conventions}
To state the main results of the paper, let us first introduce some
relevant definitions and set our conventions and notation.

Throughout the paper, we assume that $(M^{2n}, \omega)$ is a closed
\emph{toroidally monotone} or \emph{toroidally negative monotone}
symplectic manifold unless specifically stated otherwise. To be more
specific, recall that a cohomology class $w$ is \emph{atoroidal} if
for every map $v\colon \T^2\to M$ the integral of $w$ over $v$
vanishes: $\left< w, [v] \right>=0$. A symplectic manifold
$(M,\omega)$ is said to be toroidally monotone (resp., toroidally
negative monotone) when for some constant $\lambda\geq 0$ (resp.,
$\lambda<0$) the class $w=[\omega]-\lambda c_1(TM)$ is atoroidal. The
constant $\lambda$ is referred to as the \emph{toroidal monotonicity
  constant}.  Note that the case $\lambda=0$ corresponds to an
atoroidal class $[\omega]$.  Toroidally (negative) monotone manifolds
are automatically spherically (negative) monotone. We refer the reader
to Section \ref{sec:res-homology} for examples of toroidally monotone
or negative monotone manifolds. We call the positive generator $\NT$
of the group generated by the integrals $\left< c_1(TM), [v] \right>$
for all tori the \emph{minimal toroidal first Chern number} of $M$. We
set $\NT=\infty$ when this group is $\{0\}$, i.e., $c_1(TM)$ is
atoroidal. For a toroidally monotone or negative monotone manifold,
this implies that $[\omega]$ is also atoroidal.

We denote by $\tpi_1(M)$ the set of homotopy classes of free loops in
$M$. The free homotopy class of a loop $x$ and its integer homology
class (modulo torsion) are denoted by $\llbracket x \rrbracket$ and,
respectively, by $[x]$. Likewise, we write $[\ff]\in \H_1(M;\Z)/\Tor$
for the homology class modulo torsion of a free homotopy class
$\ff\in \tpi_1(M)$.

All Hamiltonians $H$ are assumed to be one-periodic in time, i.e.,
$H\colon S^1\times M\to\R$, and we set $H_t = H(t,\cdot)$ for
$t\in S^1=\R/\Z$.  The Hamiltonian vector field $X_H$ of $H$ is
defined by $i_{X_H}\omega=-dH$. The (time-dependent) flow of $X_H$ is
denoted by $\varphi_H^t$ and its time-one map by $\varphi_H$. Such
time-one maps are called \emph{Hamiltonian diffeomorphisms}. For the
sake of brevity, we will refer to the periodic orbits of $\varphi_H$
or, equivalently, the periodic orbits of $\varphi_H^t$ with integer
period as the periodic orbits of $H$. For a Hamiltonian $H$ and a
collection of free homotopy classes $\fc\subset \tpi_1(M)$, we set
$\PP_k(H, \fc)$ to be the set of $k$-periodic orbits of $H$ in
$\fc$. For instance, $\PP_k(H, \gamma)$, where
$\gamma \in \H_1(M;\Z)/\Tor$, comprises the $k$-periodic orbits of $H$
in the homology class~$\gamma$.

For a class $\ff\in \tpi_1(M)$, let us fix a reference loop
$z\in \ff$.  A capping of $x \colon S^1 \to M$ with free homotopy
class $\ff$ is a cylinder (i.e., a homotopy)
$\Pi \colon [0,1] \times S^1 \to M$ connecting $x$ and $z$ taken up to
a certain equivalence relation. Namely, two cappings $\Pi$ and $\Pi'$
are equivalent if the integral of $c_1(TM)$, and hence of $\omega$,
over the torus obtained by attaching $\Pi'$ to $\Pi$ is equal to zero.

The \emph{action} of $H$ on a capped loop $\bx=(x,\Pi)$ is
$$
\CA_H(\bx)=
-\int_\Pi\omega+\int_{S^1} H_t(x(t))\,dt.
$$
Clearly, $\CA_H(\bx)$ is well defined. Moreover, the critical points
of $\CA_H$ are exactly the capped one-periodic orbits of $H$ in the
homotopy class $\ff$.  The action spectrum $\CS(H,\ff)$ is the set of
critical values of $\CA_H$. It has zero measure; see, e.g.,~\cite{HZ}.

Furthermore, let us fix a trivialization of $TM$ along the reference
loop $z$. Then, to a capped one-periodic orbit $\bx$ with $x\in\ff$,
one can associate the \emph{mean index} $\Delta_H(\bx)$ in a standard
way. Namely, we extend the trivialization of $TM|_z$ to the capping of
$x$ and then use the resulting trivialization of $TM|_x$ to turn the
linearized flow $d\varphi_H^t|_x$ along $x$ into a path in the group
$\Sp(2n)$. The mean index $\Delta_H(\bx)$ is by definition the mean
index of the resulting path; see, e.g., \cite{Lo,SZ}. The mean index
measures, roughly speaking, the total rotation number of certain unit
eigenvalues of the linearized flow along $x$. The mean index
$\Delta_H(x)$ of a non-capped orbit $x$ is well defined as an element
of $\R/2\NT\Z$.

By analogy with the case of contractible orbits (see \cite{GG:gaps}),
we define the \emph{augmented action} of a one-periodic orbit $x$ to
be
$$
\tA_H(x)=\CA_H(\bx)-\frac{\lambda}{2}\Delta_H(\bx).
$$
The action and the mean index change under recapping in the same way,
up to the factor $\lambda/2$, and hence the augmented action of $x$ is
well defined, i.e., independent of the capping. Note, however, that
the augmented action depends on the choices of the reference loop $z$
and the trivialization. When $\lambda=0$, i.e., $[\omega]$ is
atoroidal the augmented action turns into the ordinary action.

The \emph{augmented action spectrum} $\tCS(H,\ff)$ is the collection of
the augmented action values for all one-periodic orbits of $H$ in the
class $\ff$, i.e.,
$$
\tCS(H,\ff) = \{ \tA_H(x) \mid x \in \PP_1(H, \ff)\}.
$$
In contrast with the ordinary action spectrum, $\tCS(H,\ff)$ need not
have zero measure -- in fact, it can contain whole intervals --
unless, of course, $H$ has finitely many periodic orbits in
$\ff$. However, $\tCS(H,\ff)$ is a compact set, which depends
continuously on $H$. To be more precise, as is easy to see, for any
open set $U\supset \tCS(H,\ff)$, we have $U\supset \tCS(K,\ff)$ when
$K$ is sufficiently $C^1$-close to $H$.

Assume now that all one-periodic orbits of $H$ in the class $\ff$ with
augmented action in an open interval $I$ are isolated. Then we set
$\chi(H,I,\ff)$ to be the sum of the Poincar\'e-Hopf indices of their
return maps. This definition extends by continuity to all Hamiltonians
$H$ with possibly non-isolated orbits as long as the end points of $I$
are outside $\tCS(H,\ff)$. For instance, $\chi(H,I,\ff)\neq 0$ when
$I\cap \tCS(H,\ff)$ contains only one point $a$ and $H$ is
non-degenerate and has an odd number of one-periodic orbits (e.g.,
one) in the class $\ff$ with augmented action $a$.

The above definitions generalize in an obvious way to the setting
where one-periodic orbits are replaced by $k$-periodic orbits or where
one class $\ff\in\tpi_1(M)$ is replaced by a collection of classes
$\fc\subset \tpi_1(M)$. In the latter case, a reference loop $z$ and a
trivialization of $TM|_z$ must be fixed for every $\ff\in\fc$.

\subsection{Results} 
\label{sec:res-homology}
Now we are in a position to state our main results.  We recall that
\emph{all manifolds considered in this paper are assumed to be
  toroidally monotone or toroidally negative monotone unless
  specifically stated otherwise}.  Furthermore, since the proofs
heavily depend on Hamiltonian Floer homology, we need to either assume
throughout the paper that $M$ is weakly monotone or rely on the
machinery of virtual cycles. To be more precise, recall that a
manifold $M^{2n}$ is said to be weakly monotone if one of the
following conditions is satisfied: $M$ is (not necessarily strictly)
monotone, i.e., $[\omega]|_{\pi_2(M)}=\lambda c_1(TM)|_{\pi_2(M)}$
with $\lambda\geq 0$, or $c_1(TM)|_{\pi_2(M)}=0$, or $N\geq n-2$ where
$N$ is the minimal Chern number. Under any of these conditions, the
Floer homology is defined and has standard properties; see \cite{HS,
  MS, Ono} and references therein. We refer the reader to, e.g.,
\cite{FO, HWZ, LT, Pa} for various incarnations of the technique of
virtual cycles. In all but one of our results (Theorem
\ref{thm:generic}) the ambient manifold $M$ is automatically monotone
(not strictly), and hence weakly monotone.

Our first result asserts that under certain additional assumptions on
$M$ the presence of one non-contractible hyperbolic one-periodic orbit
implies the existence of infinitely many non-contractible periodic
orbits; cf.\ \cite{GG:hyperbolic}.

\begin{Theorem}
\label{thm:homology-hyperbolic}
Assume that $\NT\geq n/2+1$, where $2n=\dim M$, and that a Hamiltonian
$H$ on $M$ has a hyperbolic one-periodic orbit $x$ such that all
homotopy classes in the set
$\llbracket x \rrbracket^\N=\{\llbracket x \rrbracket^k\mid k\in\N\}$
are distinct and nontrivial. (This is the case, e.g., when $[x]\neq 0$
in $\H_1(M;\Z)/\Tor$.)  Then $H$ has infinitely many simple periodic
orbits with homotopy class in $\llbracket x \rrbracket^\N$. In
particular, if all such orbits are isolated, there are simple
non-contractible periodic orbits of arbitrarily large period.
\end{Theorem}

The condition that $\NT\geq n/2+1$ appears to be purely technical even
though it plays an essential role in the proof. As we will show below,
there are numerous symplectic manifolds $M$ and Hamiltonians $H$
meeting the requirements of the theorem. For instance, the theorem
applies to $M=\Sigma_g\times \CP^m$, where $\Sigma_g$ is a closed
surface of genus $g\geq 2$.

The second result is more accurate and covers a broader range of
manifolds and maps, although it still relies on some additional
topological assumptions about the flow of $H$; cf.\ \cite{Gu:nc}.

\begin{Theorem}
\label{thm:homology-chi}
Assume that $\PP_1(H,[\ff])$ is finite and $\chi(H,I,\ff)\neq 0$ for
some interval $I$ with end points outside $\tCS(H,\ff)$, where
$\ff\in\tpi_1(M)$ and $[\ff] \neq 0$ in $\H_1(M;\Z)/\Tor$.  Then, for
every sufficiently large prime $p$, the Hamiltonian $H$ has a simple
periodic orbit in the homotopy class $\ff^p$ and with period either
$p$ or $p'$, where $p'$ is the first prime greater than $p$. Moreover,
when $\pi_1(M)$ is hyperbolic and torsion free, the condition
$[\ff]\neq 0$ can be replaced by $\ff\neq 1$ and no finiteness
requirement is needed.
\end{Theorem}

\begin{Remark}
  We emphasize that in these theorems we impose no non-degeneracy
  conditions on $H$. An interesting new point in the second part of
  Theorem \ref{thm:homology-chi} is that, in contrast with many other
  results of this type, we do not need to require $\PP_1(H,\ff)$ to be
  finite to have simple periodic orbits with arbitrarily large period.
  It is immediate to see that, if $H$ is non-degenerate,
  $\chi(H,I,\ff)\neq 0$ for a short interval $I$ centered at
  $a\in\tCS(H,\ff)$ when the one-periodic orbits of $H$ in the class
  $\ff$ have different augmented action values.

  Also note that by passing to an iteration one can replace in both of
  the theorems one-periodic orbits by $k$-periodic. Then the first
  theorem remains correct as stated. In the second theorem, after
  replacing $H$ by the iterated Hamiltonian $H^{\nat k}$ (see Section
  \ref{sec:prelim}), we can only conclude that $H$ has infinitely many
  simple periodic orbits in $\ff^\N$. (Of course, the theorem still
  applies literally to $H^{\nat k}$ in place of $H$, but simple
  periodic orbits of $H^{\nat k}$ are not necessarily simple as
  periodic orbits of $H$.)
\end{Remark}

Let us now further discuss the conditions of Theorems
\ref{thm:homology-hyperbolic} and \ref{thm:homology-chi}, beginning
with those concerning the manifold and then moving on to the
Hamiltonian.

The manifolds $M$ meeting the requirements of these theorems exist in
abundance.  To construct specific examples, let us start with a
symplectically atoroidal manifold $(M_1,\omega_1)$, i.e., a closed
symplectic manifold such that $[\omega_1]$ and $c_1(TM_1)$ are
atoroidal. Among these are, for instance, surfaces of genus $g\geq 2$
and, more generally, all K\"ahler manifolds with negative sectional
curvature. (See, e.g., \cite{Gu:nc} for further references and a
discussion of such manifolds.)  Next, let $(M_0,\omega_0)$ be a closed
spherically monotone or negative monotone symplectic manifold. There
are numerous examples of such manifolds including, in the negative
monotone case, those with arbitrarily large minimal (spherical) Chern
number $\NS$. For instance, let $M_0$ be a smooth complete
intersection in $\CP^{m+k}$ given by $k$ homogeneous polynomials of
degrees $d_1,\ldots,d_k$. Then $M_0$ is monotone or negative monotone
with $\NS=|m+k-d|$, where $d=d_1+\ldots+d_k$, unless $\NS=0$ (see,
e.g., \cite[p.\ 88]{LM}). To be more precise, $M_0$ is monotone when
$m+k-d>0$, negative monotone when $m+k-d<0$, and $c_1(TM_0)=0$ when
$m+k-d=0$. Now it is easy to see that $M=M_0\times M_1$ is toroidally
monotone or toroidally negative monotone (when $\NS\neq 0$) with
$\NT=\NS$.

Furthermore, $\pi_1(M)=\pi_1(M_0)\times \pi_1(M_1)$. In particular,
$\H_1(M;\Z)/\Tor\neq 0$ when $\H_1(M_1;\Z)/\Tor\neq 0$. Moreover, when
$M_0$ is a complete intersection and $M_1$ is a K\"ahler manifold with
negative sectional curvature, $\pi_1(M)$ is hyperbolic and torsion
free. Indeed, in this case, $\pi_1(M)=\pi_1(M_1)$ since complete
intersections are simply connected (see, e.g., \cite[Chap.\ IX, Sect.\
4.1]{Sha}).

Next, note that for any symplectic manifold $M$ and $\ff\in \tpi_1(M)$
there exists a Hamiltonian $H\colon S^1\times M\to \R$ with a
hyperbolic one-periodic orbit in the class $\ff$; see, e.g.,
\cite[Prop.\ 1.3]{Ba}. (In fact, one can prescribe arbitrarily a
periodic orbit and the linearization of the flow along it.) Thus,
whenever $M$ satisfies the conditions of Theorem
\ref{thm:homology-hyperbolic} and $\H_1(M;\Z)/\Tor\neq 0$ or
$\pi_1(M)$ is hyperbolic and torsion free, there exists a
$C^\infty$-open, non-empty set of Hamiltonians this theorem applies
to. Furthermore, in the setting of Theorem \ref{thm:homology-chi},
the collection of Hamiltonians with one-periodic orbits in $\ff$ is
non-empty and, as one can easily see, has a locally non-empty
interior, i.e., the intersection of this collection with a
neighborhood of its any point has a non-empty interior.  It is also
not hard to show that the Hamiltonians meeting the requirements of the
theorem form a $C^\infty$-open and dense subset in this
collection. (Indeed, non-degenerate Hamiltonians form a
$C^\infty$-open and dense subset, and one can further ensure by a
$C^\infty$-small perturbation of $H$ that all one-periodic orbits have
distinct augmented actions.)

It is worth pointing out that it is not clear how large this open
subset is, i.e., how common the Hamiltonians with at least one
non-contractible periodic orbit are. This is an interesting question
and to the best of our understanding very little is known about this
problem. Consider, for instance, time-dependent Hamiltonians on a
closed surface $M$ of positive genus. Then a bump function with small
support or, as Paul Seidel pointed out to us, a self-indexing Morse
function are examples of Hamiltonians without non-contractible
periodic orbits of any period. Furthermore, a simple KAM argument
shows that for a $C^\infty$-generic autonomous Hamiltonian $H$ on
$M=\T^2$ such that one of the components of the level $\{H=0\}$ is a
parallel, no Hamiltonian sufficiently $C^\infty$-close to $H$ has
periodic orbits in the free homotopy class collinear to the class of
the meridian. (This observation is due to Leonid Polterovich.)
However, as far as we know there are no counterexamples to the
conjecture that a $C^1$-generic (or even $C^\infty$-generic)
Hamiltonian on $M$ has a non-contractible periodic orbit.  In fact, as
has been pointed out to us by Patrice Le Calvez and Andr\'es
Koropecki, the conjecture holds for $M=\T^2$ for $C^\infty$-generic
Hamiltonians, \cite[Prop.\ J]{LCT}. (See also \cite{Ta} for some
possibly relevant results.)

Our next theorem is a minor refinement of \cite[Thm.\
1.1]{Gu:nc}. This is a result stronger than Theorems
\ref{thm:homology-hyperbolic} and \ref{thm:homology-chi}, but
applicable to a much more narrow class of manifolds.

\begin{Theorem}
\label{thm:atoro}
Assume that the class $[\omega]$ is atoroidal and let $H$ be a
Hamiltonian having a non-degenerate one-periodic orbit $x$ with
homotopy class $\ff$ such that $[\ff] \neq 0$ in $\H_1(M;\Z)/\Tor$ and
$\Pp_1(H,[\ff])$ is finite.  Then, for every sufficiently large prime
$p$, the Hamiltonian $H$ has a simple periodic orbit in the homotopy
class $\ff^p$ and with period either $p$ or $p'$, where $p'$ is the
first prime greater than $p$. Moreover, when $\pi_1(M)$ is hyperbolic
and torsion free, the condition $[\ff]\neq 0$ can be replaced by
$\ff\neq 1$.
\end{Theorem}

\begin{Remark}
  Here, as in \cite[Thm.\ 3.1]{Gu:nc}, the requirement that $x$ is
  non-degenerate can be replaced by a much less restrictive condition
  that $x$ is isolated and has non-trivial local Floer homology; see
  Theorem \ref{thm:atoro2}. The key difference between Theorems
  \ref{thm:atoro} and \ref{thm:homology-chi} is that a non-degenerate
  Hamiltonian $H$ can, at least hypothetically, have several
  one-periodic orbits in the class $\ff$, yet $\chi(H,I,\ff) =0$ for
  any interval $I$.  (The reason why in Theorem
  \ref{thm:homology-chi}, in contrast with Theorem \ref{thm:atoro} or
  \cite{Gu:nc}, it is not sufficient to assume that $H$ has a
  non-contractible periodic orbit is that, as has already been
  mentioned in the introduction, the augmented action filtration is
  not strict. We will come back to this issue in Section
  \ref{sec:filtration}.)  The new points of Theorem \ref{thm:atoro} as
  compared with the results from \cite{Gu:nc} are the ``moreover''
  part of the theorem and the control of the homotopy classes of the
  orbits rather than the homology classes.
  \end{Remark}

  Among symplectic manifolds with atoroidal class $[\omega]$ are the
  K\"ahler manifolds with negative sectional curvature
  mentioned above and also some other classes of symplectic manifolds;
  see, e.g., \cite{BK,Ked}.

\begin{Remark}[Growth]
  In the settings of Theorems \ref{thm:homology-chi} and
  \ref{thm:atoro}, the number of simple non-contractible periodic
  orbits with period less than or equal to $k$, or the number of
  distinct homotopy classes represented by such orbits, is bounded
  from below by $\const \cdot k/ \ln k$.  An immediate consequence of
  the theorems is that $H$ has infinitely many simple periodic orbits
  with homology class in $\N[\ff]$ regardless of whether or not the
  set of one-periodic orbits (in the class $[\ff]$) is finite.
\end{Remark}

The simplest manifold the above theorems do not apply to is the
standard symplectic torus $\T^{2n}$ with $2n\geq 4$. In dimension two,
it is easy to see that for strongly non-degenerate Hamiltonian
diffeomorphisms the presence of one non-contractible orbit in a
homotopy class $\ff$ implies the existence of infinitely many simple
periodic orbits with homotopy class in $\ff^\N$. (Recall that a
diffeomorphism is said to be strongly non-degenerate if all its
periodic orbits are non-degenerate.) The proof of this fact amounts to
the observation that in dimension two the mean index determines the
Conley--Zehnder index and is similar to the proofs of \cite[Thm.\
5.1.9]{Ab} or \cite[Thm.\ 1.7]{GG:generic}. However, somewhat
surprisingly, it is not clear at all whether the non-degeneracy
condition here can be replaced as in, e.g., \cite{Gu:nc, Gu:hq} by the
requirement that the orbit has non-vanishing local Floer homology.
When $2n\geq 4$, no results along these lines have been established
for $\T^{2n}$.

It is interesting to contrast the previous theorems with the
following, admittedly superficial, $C^\infty$-generic existence
result.  Namely, $C^\infty$-generically, the existence of one
non-contractible one-periodic orbit is sufficient to infer the
existence of infinitely many simple non-contractible periodic orbits under
no conditions on $M$ and with only very minor assumptions about $\ff$;
cf. \cite{GG:generic}. To be more precise, denote by $\FF_\ff$ the
collection of strongly non-degenerate Hamiltonian diffeomorphisms with
a one-periodic orbit in $\ff$. Clearly, $\FF_\ff$ is $C^\infty$-open
in the group of Hamiltonian diffeomorphisms. As has been pointed out
above, this set is non-empty by, e.g., \cite[Prop.\ 1.3]{Ba}. Recall
also that a subset is called residual, or second Baire category, when
it is the intersection of a countable collection of open and dense
subsets.

\begin{Theorem}
\label{thm:generic}
Assume that $\ff^k\neq 1$ for all $k\in\N$. Then the subset
$\FF_\ff^\infty$ of $\FF_\ff$ formed by Hamiltonian diffeomorphisms
with infinitely many simple periodic orbits in $\ff^\N$ is
$C^\infty$-residual.
\end{Theorem}

It is essential that in this theorem the ambient manifold $M$ is not
required to be toroidally monotone or negative monotone. In fact, no
conditions on $M$, other than compactness, is necessary. However, as
in the case of other results of this paper, the proof makes use of the
Hamiltonian Floer homology, and we need to either assume that $M$ is
weakly monotone or rely on the machinery of virtual cycles.

A consequence of Theorem \ref{thm:generic} in the spirit of an
observation in \cite{PS} is that non-autonomous Hamiltonian
diffeomorphisms (i.e., Hamiltonian diffeomorphisms that cannot be
generated by autonomous Hamiltonians) form a $C^\infty$-residual
subset of $\FF_\ff$. In fact, every $\varphi\in\FF_\ff^\infty$ is
necessarily non-autonomous. Indeed, when $k>1$, simple $k$-periodic
orbits of an autonomous Hamiltonian diffeomorphism are never isolated,
and hence, in particular, never non-degenerate.

\section{Augmented action filtration}
\label{sec:augfil}
Our goal in this section is to show that when the augmented action gap
is sufficiently large the Floer homology for non-contractible periodic
orbits is filtered by the augmented action, and to analyze the
behavior of this filtration under continuation maps.  As in the rest
of the paper, we assume that $(M^{2n}, \omega)$ is a closed,
toroidally monotone or toroidally negative monotone symplectic
manifold. However, the construction of the Floer homology (but not of
the augmented action filtration) goes through in general for any
compact manifold $M$, at least when $M$ is weakly monotone or via the
technique of virtual cycles.

\subsection{Preliminaries: iterated Hamiltonians}
\label{sec:prelim}

Let $H\colon S^1\times M\to\R$ be a one-periodic in time Hamiltonian
on $M$.  The augmented action of $H$ is homogeneous under the
iterations of $\varphi_H$. To make this more precise, let us recall a
few standard definitions.

Let $K$ and $H$ be two one-periodic Hamiltonians. The ``composition''
$K\nat H$ is, by definition, the Hamiltonian
$$
(K\nat H)_t=K_t+H_t\circ(\varphi^t_K)^{-1},
$$
and the flow of $K\nat H$ is $\varphi^t_K\circ \varphi^t_H$. We set
$H^{\nat k}=H\nat\ldots \nat H$ ($k$ times).  Abusing terminology, we
will refer to $H^{\nat k}$ as the $k$th iteration of $H$.  (Note that
the flow $\varphi^t_{H^{\nat k}}=(\varphi_H^t)^k$, $t\in [0,\,1]$, is
homotopic with fixed end-points to the flow $\varphi^t_H$,
$t\in [0,\, k]$.)

In general, $H^{\nat k}$ is not one-periodic, even when $H$
is. However, $H^{\nat k}$ becomes one-periodic when, for example,
$H_0\equiv 0\equiv H_1$. The latter condition can always be met by
reparametrizing the Hamiltonian as a function of time without changing
the time-one map. This procedure does not affect the Hofer norm, and
actions and indices of the periodic orbits.  Thus, in what follows, we
usually treat $H^{\nat k}$ as a one-periodic
Hamiltonian. Alternatively, the Hamiltonian diffeomorphism
$\varphi_H^k$ can be obtained as the time-$k$ flow of $H$. Thus, in
some instances such as the proof of Lemma \ref{lemma:hom}, it is more
convenient to treat $H^{\nat k}$ as the $k$-periodic Hamiltonian $H_t$
with $t\in\R/k\Z$. We will always state specifically when this is the
case. Clearly, these two Hamiltonians, both denoted by $H^{\nat k}$,
have canonically isomorphic filtered Floer homology.

The $k$th iteration of a one-periodic orbit $x$ of $H$ is denoted by
$x^k$.  More specifically, $x^k$ is the $k$-periodic orbit $x(t)$,
$t\in [0,\,k]$, of $H$.  There is an action-- and mean
index--preserving one-to-one correspondence between the one-periodic
orbits of $H^{\nat k}$ and the $k$-periodic orbits of $H$. Thus, we
can also think of $x^k$ as the one-periodic orbit
$x^k(t)=\varphi_{H^{\nat k}}^t \left(x(0)\right)$ of $H^{\nat k}$.

Assume now that all iterated homotopy classes $\ff^k$, $k\in\N$, are
distinct and nontrivial. As above, we have a reference loop $z\in \ff$
fixed together with a trivialization of $TM|_z$. Let us chose the
iterated loop $z^k$ with the ``iterated trivialization'' as the
reference loop for $\ff^k$. Then the action and the mean index are
both homogeneous with respect to the iteration and, as a consequence,
$$
\tA_{H^{\nat k}}(x^{k}) = k \tA_H(x).
$$

\subsection{Floer homology for non-contractible periodic orbits}
\label{sec:Floer}

The key tool used in the proofs of Theorems
\ref{thm:homology-hyperbolic} and \ref{thm:homology-chi} is the Floer
homology for non-contractible periodic orbits of Hamiltonian
diffeomorphisms.  Various flavors of Floer homology in this case for
both open and closed manifolds have been considered in several other
works; see, e.g., \cite{BPS, GL, GG:hyperbolic, Gu:nc, Lee, Ni, We}.
Below we assume that $M^{2n}$ is closed and toroidally monotone or
toroidally negative monotone. In the latter case, to have the Floer
homology defined, one must either rely on the machinery of multivalued
perturbations (and set the coefficient ring to be $\Q$) or require in
addition that $\NS \geq n$ to ensure that $M$ is weakly monotone,
where $\NS$ is the minimal spherical Chern number; cf.\ \cite{LO}.

Let us now briefly describe the elements of the construction of the
Floer homology relevant to our argument. Fix $\ff\in\tpi_1(M)$.  Let
$H$ be a Hamiltonian such that all one-periodic orbits of $H$ in $\ff$
are non-degenerate. (Here $x$ is said to be non-degenerate if the
linearized return map $d\varphi_H \colon T_{x (0)}M\to T_{x (0)}M$
does not have one as an eigenvalue.)  The Floer complex $\CF(H,\ff)$
is generated, over some fixed coefficient ring, by these orbits. The
Floer differential is defined in the standard way. With this
definition, the complex $\CF(H,\ff)$ is neither graded nor does it
carry an action filtration.  The homology $\HF(H,\ff)$ of $\CF(H,\ff)$
is equal to zero when $\ff\neq 1$. Indeed, by the standard
continuation argument $\HF(H,\ff)$ is independent of $H$ (cf.\ Section
\ref{sec:cont}) and, since all one-periodic orbits of a $C^2$-small
autonomous Hamiltonian $H$ are contractible, we have
$\HF(H,\ff)=0$. As is well known, $\HF(H,1)=\H_*(M)$ at least over
$\Q$; see, e.g., \cite{MS} for further references.

To give the complex $\CF(H,\ff)$ some more structure, let us fix a
reference loop $z\in\ff$ and a trivialization of $TM|_z$.  Using this
trivialization, we can define the \emph{Conley--Zehnder index}
$\MUCZ(H,\bx)\in\Z$ of a capped \emph{non-degenerate} orbit $\bx$ as
in, e.g., \cite{MS,SZ}.  For future reference, note that
\begin{equation}
\label{eq:mean-CZ}
|\Delta_H(\bx)-\MUCZ(\bx)|\leq n .
\end{equation}

Similarly to the contractible case, the Conley-Zehnder $\MUCZ(H,x)$ of
an orbit without capping is defined only modulo $2\NT$. As a result,
we obtain a $\Z_{2\NT}$-grading of the complex $\CF(H,\ff)$ and of the
homology $\HF(H,\ff)$, and, in particular, a $\Z_2$-grading. Replacing
the one-periodic orbits of $H$ by the capped one-periodic orbits, one
could define the Floer complex and the homology of $H$ as a module
over a suitably chosen Novikov ring and, as in the contractible case,
this complex and the homology would be $\Z$-graded and filtered by the
action. However, for our purposes it is more convenient to work with
the homology $\HF(H,\ff)$ and the complex $\CF(H,\ff)$ generated by
the non-capped orbits and defined as above.

The constructions from this section readily carry over to the case
where a single free homotopy class $\ff$ is replaced by a collection
of free homotopy classes. For instance, one can specify the collection
of free homotopy classes of loops by prescribing a homology class.

\subsection{Filtration}
\label{sec:filtration}
Let, as above, $M^{2n}$ be a toroidally monotone or toroidally
negative monotone closed symplectic manifold with monotonicity
constant $\lambda$.  In what follows, we have a free homotopy class
$\ff$ or a collection of such classes together with the reference
loops and trivializations fixed and suppressed in the notation. Thus
we write $\CF(H)$ for $\CF(H,\ff)$, etc. With these auxiliary data
fixed, the augmented action spectrum $\tCS(H):=\tCS(H,\ff)$ is defined
for any Hamiltonian $H$ on $M$.

The \emph{augmented action gap} is the infimum of the distance between
two distinct points in the augmented action spectrum $\tCS(H)$, i.e.,
$$
\gap(H)=\inf |s-s'|\in [0,\,\infty], \textrm{ where $s$ and
  $s'\neq s$ are in $\tCS(H)$.}
$$
We emphasize that $\gap(H)$ is defined even when $H$ is degenerate. It
is also worth pointing out that $\gap(H)$ is neither upper nor lower
semicontinuous in $H$.

Set
\begin{equation}
\label{eq:C0}
c_0(M)=|\lambda|\frac{2n\pm 1}{2},
\end{equation}
where the sign $\pm$ is $\sign(\lambda)$. We say that the \emph{gap
  condition} is satisfied whenever
\begin{equation}
\label{eq:gap}
\gap(H)>c_0(M).
\end{equation}

\begin{Proposition}
\label{prop:action-filt}
Assume that all one-periodic orbits of $H$ in $\ff$ are non-degenerate
and \eqref{eq:gap} holds. Then the complex $\CF(H)$, and hence the
homology $\HF(H)$, is filtered by the augmented action. In other
words,
\begin{equation}
\label{eq:action-filt}
\tA_H(y)\leq \tA_H(x)
\end{equation}
whenever $y$ enters $\p x=\sum a_y y$ with non-zero coefficient.
\end{Proposition}

\begin{Remark}
  In contrast with the standard action filtration, the augmented
  action filtration is not necessarily strict, i.e., equality in
  \eqref{eq:action-filt} can occur. Note also that in this proposition
  it suffices to have a non-strict inequality in the gap
  condition~\eqref{eq:gap}.
\end{Remark}

\begin{proof} Throughout the proof, let us assume that
  $\lambda\geq 0$, i.e., $M$ is toroidally monotone. The negative
  monotone case is dealt with by a similar (up to some signs)
  calculation.

  To establish \eqref{eq:action-filt}, let us fix a capping of
  $x$. Then an orbit $y$ entering $\p x$ with non-zero coefficient
  inherits a capping from $\bx$. We have

 \begin{equation*}
\begin{split}
\tA_H(y)&= \CA_H(\by)-\frac{\lambda}{2}\Delta_H(\by)\\
&< \CA_H(\bx)-\frac{\lambda}{2}\big(\MUCZ(\by)-n\big)\\
&= \CA_H(\bx)-\frac{\lambda}{2}\big(\MUCZ(\bx)-n-1\big)\\
&\leq \CA_H(\bx)-\frac{\lambda}{2}\big(\Delta_H(\bx)-2n-1\big)\\
&= \tA_H(x)+c_0(M) .
\end{split}
\end{equation*}
Here we used \eqref{eq:mean-CZ} and the facts that
$\MUCZ(\by)=\MUCZ(\bx)-1$ and that $\p$ is action decreasing.  Thus we
have shown that $\p$ does not increase the augmented action by more
than $c_0(M)$. Now the required inequality \eqref{eq:action-filt}
follows once the augmented action gap is greater than $c_0(M)$, i.e.,
when the gap condition \eqref{eq:gap} holds.
\end{proof}

With Proposition \ref{prop:action-filt} in mind, we can define the
augmented action filtration on the homology exactly in the same way as
for the ordinary action.  Thus, assume that \eqref{eq:gap} is
satisfied and $a\not\in \tCS(H)$ and denote by
$\tCF^{(-\infty,\,a)}(H)$ the subcomplex of $\CF(H)$ generated by the
orbits with augmented action below $a$.  Let $\tHF^{(-\infty,\,a)}(H)$
be the homology of this subcomplex.  Furthermore, when $I=(a,\,b)$ is
an interval with end points outside $\tCS(H)$, we set
$$
\tCF^{I}(H)=\tCF^{(-\infty,\,b)}(H)/ \tCF^{(-\infty,\,a)}(H).
$$
In other words, $\tCF^{I}(H)$ is the complex generated by the orbits
with augmented action in $I$, equipped with the naturally defined
differential. We denote the homology of this complex by $\tHF^{I}(H)$.
(The role of the tilde here is to emphasize that we use the augmented
action rather than the ordinary action and that $I$ is an augmented
action range.) We have the long exact sequence
$$
\cdots\to \tHF^{(-\infty,\,a)}(H)\to \tHF^{(-\infty,\,b)}(H)\to
\tHF^{I}(H) \to\cdots
$$
and a similar exact sequence for three intervals
\begin{equation}
\label{eq:exact-seq}
\cdots\to \tHF^{(c,\,a)}(H)\to \tHF^{(c,\,b)}(H)\to \tHF^{(a,\,b)}(H)
\to\cdots .
\end{equation}

Our next goal is to show that the construction of the augmented action
filtered Floer homology extends by continuity to all, not necessarily
non-degenerate, Hamiltonians.

\begin{Proposition}
\label{prop:action-filt2}
Let $H$ be a Hamiltonian on $M$, not necessarily non-degenerate, such
that the gap condition \eqref{eq:gap} is satisfied and let
$a\not\in \tCS(H)$. Then for any non-degenerate Hamiltonian $K$
sufficiently $C^1$-close to $H$, the subspace
$\tCF^{(-\infty,\,a)}(K)\subset \CF(K)$ is a subcomplex.
\end{Proposition}

This result is not an immediate consequence of Proposition
\ref{prop:action-filt}. Since the augmented action gap is not lower
semicontinuous in the Hamiltonian, we cannot guarantee that
\eqref{eq:gap} holds for $K$ if it holds for $H$, and thus \emph{a
  priori} Proposition \ref{prop:action-filt} need not apply to $K$.

\begin{proof} Let $x$ be a one-periodic orbit of $K$ with $\tA_K(x)<a$
  and let $y$ be an orbit entering $\p x$ with non-zero
  coefficient. We need to show that $\tA_K(y)<a$.

  The orbits $x$ and $y$ are $C^1$-small perturbations of one-periodic
  orbits $x'$ and $y'$ of $H$ with augmented actions close to those of
  $x$ and $y$.  By continuity of the augmented action spectrum, we
  necessarily have $\tA_H(x')<a$ when $K$ is $C^1$-close to $H$.

  If $\tA_H(y')>\tA_H(x')$, we have $\tA_H(y')-\tA_H(x')> c_0(M)$ by
  \eqref{eq:gap}, and therefore $\tA_K(y)-\tA_K(x)> c_0(M)$. This is
  impossible because, as we have seen from the proof of Proposition
  \ref{prop:action-filt}, the differential cannot increase the
  augmented action by more than $c_0(M)$.  Thus
  $\tA_H(y')\leq\tA_H(x')$. Then
$$
\tA_K(y)\approx \tA_H(y')\leq \tA_H(x')<a,
$$ 
and hence $\tA_K(y)<a$ when $K$ is $C^1$-close to $H$.
\end{proof}

Now, for any Hamiltonian $H$, when the end-points of an interval $I$
are outside $\tCS(H)$ and \eqref{eq:gap} holds, we can, utilizing
Proposition \ref{prop:action-filt2}, set $\tHF^I(H):=\tHF^I(K)$, where
$K$ is a $C^1$-small non-degenerate perturbation of $H$. Using
standard continuation arguments (cf.\ Section \ref{sec:cont}), it is
easy to see that the resulting homology is well defined, i.e.,
independent of $K$.

\begin{Example}
  The setting we are interested in where the gap condition
  \eqref{eq:gap} is satisfied is when $H$ is a high prime order
  iteration of some Hamiltonian $F$, i.e., $H=F^{\nat k}$ and $k$ is a
  large prime. In this case, either $F$ has simple $k$-periodic orbits
  or $\gap(H)=k\gap(F)$. Thus, either new periodic orbits are created
  or the gap grows linearly under the iterations of $F$ and eventually
  becomes greater than $c_0(M)$.
\end{Example}

It is clear that all these constructions respect the
$\Z_{2\NT}$-grading (and hence the $\Z_2$-grading) of the complexes
and the homology. Thus, for instance, \eqref{eq:exact-seq} is an exact
sequence of graded complexes and the connecting map has degree $-1$.

For degenerate Hamiltonians with isolated one-periodic orbits, one
can, similarly to the case of the standard action filtration, view the
local Floer homology as building blocks for the Floer homology
filtered by the augmented action. For instance, assume that
$\tCS(H)\cap I=\{c\}$ and all one-periodic orbits $x$ with augmented
action $c$ are isolated. Then it is not hard to see that there exists
a spectral sequence with $E^1=\bigoplus_x\HF(x)$ converging to
$\tHF^I(H)$, where $\HF(x)$ stands for the local Floer homology of
$x$. (We refer the reader to, e.g., \cite{Gi:CC,GG:gap,McL} for the
definition and a discussion of the local Floer homology.) In contrast
with the case of the ordinary action filtration, we do not necessarily
have $E^1= \tHF^I(H)$ even when $H$ is non-degenerate. The reason is
that the augmented action filtration is not strict and the Floer
differential, or more generally Floer trajectories, can connect orbits
with equal augmented action.

However, as is easy to see, for any interval $I$ with end points
outside $\tCS(H)$, we have
\begin{equation}
\label{eq:chi-hom}
\chi(H,I)=(-1)^n\big[\dim \tHF^I_{\mathit{even}}(H)
-\dim \tHF^I_{\mathit{odd}}(H)\big].
\end{equation}
In particular, $\tHF^I(H)\neq 0$ if $\chi(H,I)\neq 0$. (Here, as
everywhere in this section, we have suppressed the class $\ff$ in the
notation.)

\subsection{Homotopy and continuation}
\label{sec:cont}

The behavior of the augmented action under homotopies is similar to
that of the ordinary action. Namely, recall that a continuation map
shifts the action filtration upward by a certain constant; see, e.g.,
\cite[Sect.\ 3.2.2]{Gi:coiso}. This is still true for the augmented
action, although the size of the shift is slightly
different. Furthermore, when the homotopy is monotone decreasing, the
action shift is zero, and the induced map in homology preserves the
action filtration. This fact does not have a direct analogue for the
augmented action, but the augmented action filtration is preserved
when the augmented action gaps for the Hamiltonians are large enough.

To be more precise, consider a homotopy $H^s$ from a Hamiltonian $H^0$
to a Hamiltonian $H^1$ on $M$, and set
$$
c_a(H^s)=\int_{-\infty}^\infty\int_{S^1}\max_M \p_s H^s\,dt\,ds.
$$
For instance,
$$
c_a(H^s)=\int_{S^1}\max_M (H^1-H^0)\,dt\
$$
when $H^s$ is a linear homotopy from $H^0$ to $H^1$. The augmented
action shift is governed by the constant
\begin{equation}
\label{eq:shift}
c_h(H^s)=\max\big\{0,c_a(H^s)\big\}
+|\lambda|n\geq 0.  
\end{equation}

\begin{Proposition}
\label{prop:homotopy}
Assume that both Hamiltonians $H^0$ and $H^1$ satisfy \eqref{eq:gap},
i.e.,
\begin{equation}
\label{eq:gap2}
\gap(H^0)>c_0(M)\textrm{ and } \gap(H^1)>c_0(M).
\end{equation}
Then a homotopy $H^s$ from $H^0$ to $H^1$ induces a map in the Floer
homology shifting the action filtration upward by $c_h(H^s)$:
$$
\tHF^I(H^0)\to \tHF^{c_h(H^s)+I}(H^1),
$$
where $c_h(H^s)+I$ stands for the interval $I$ moved to the right by
$c_h(H^s)$. Furthermore, if
$$
\gap(H^0)>c_h(H^s)\textrm{ and } \gap(H^1)>c_h(H^s)
$$
in addition to \eqref{eq:gap2}, the map induced by the homotopy
preserves the augmented action filtration, i.e., we have
$\tHF^I(H^0)\to \tHF^{I}(H^1)$.
\end{Proposition}

Note that here the shift $c_h(H^s)$ can be replaced by any constant
$a>c_h(H^s)$. The proof of the proposition is standard and we omit
it. Here we only mention that the first term in \eqref{eq:shift} is
the maximal action shift induced by the homotopy (see, e.g., again
\cite[Sect.\ 3.2.2]{Gi:coiso}) and the second term is the maximal mean
index shift, as can be seen from an argument similar to the proof of
Proposition \ref{prop:action-filt}.

\begin{Remark} The arguments from this section carry over to
  contractible periodic orbits, i.e., to the case where $\ff=1$, with
  some simplifications and straightforward modifications. Namely, in
  this case, it is enough to assume that $M$ is monotone or negative
  monotone to have the augmented action defined; see
  \cite{GG:gaps}. An analogue of \eqref{eq:gap} is still sufficient to
  ensure that the Floer complex and the homology are filtered by the
  augmented action and Propositions \ref{prop:action-filt},
  \ref{prop:action-filt2} and \ref{prop:homotopy} still hold.
\end{Remark}

\section{Proofs}
\label{sec:proofs}
With the action filtration introduced, we are now in a position to
prove the main results of the paper.

\subsection{Proofs of Theorems \ref{thm:homology-chi} and
  \ref{thm:atoro}}
\label{sec:prfs1}

\begin{proof}[Proof of Theorem \ref{thm:homology-chi}: the case
  $\lbrack \ff\rbrack\neq 0$]

  Since $\PP_1(H,[\ff])$ is finite, only finitely many distinct free
  homotopy classes $\ff_i\in \tpi_1(M)$ occur as the free homotopy
  classes of one-periodic orbits of $H$ in the homology class
  $[\ff]$. We claim that then, for a sufficiently large prime $p$, the
  classes $\ff_i^p$ are also distinct.

  To see this, first note that for any two elements $g\neq h$ in any
  group there is at most one prime $p$ such that $g^p=h^p$. Indeed,
  assume that there are two such distinct primes $p$ and $q$. Then,
  since $p$ and $q$ are relatively prime, $ap+bq=1$ for some integers
  $a$ and $b$. Hence,
$$
g=(g^p)^a (g^q)^b=(h^p)^a (h^q)^b=h,
$$
which is impossible since $g\neq h$. Clearly, the same is true for
conjugacy classes. As a consequence, for any finite collection of
distinct conjugacy classes, their large prime powers are also
distinct.

Throughout the proof we will always require $p$ to be a sufficiently
large prime to satisfy the above condition for the collection
$\ff_i$. (Later we will need to impose additional lower bounds on
$p$.)

Let us assume that $H$ has no simple $p$-periodic orbits in the class
$\ff^p$. Our goal is to show that it has a simple $p'$-periodic orbit,
where $p'$ is the first prime greater than $p$, in the homotopy class
$\ff^p$.

Then all $p$-periodic orbits in $\ff^p$ are the $p$th iterations of
one-periodic orbits, since $p$ is prime. Furthermore, by the above
requirement on $p$, these one-periodic orbits are necessarily in the
free homotopy class $\ff$. Thus we have
\begin{equation}
\label{eq:spec-p}
\tCS(H^{\nat p},\ff^p)=p\,\tCS(H,\ff)
\end{equation}
with respect to the $p$th iteration of the reference loop $z\in\ff$
and of the trivialization of $TM|_z$. As a consequence,
\begin{equation}
\label{eq:gap-p}
\gap(H^{\nat p},\ff^p)=p\gap(H,\ff),
\end{equation}
and the augmented action filtration on the Floer homology
$\tHF(H^{\nat p},\ff^p)$ is defined once $p$ is so large that
$p\gap(H,\ff)>c_0(M)$; see \eqref{eq:C0} and Section
\ref{sec:filtration}.  We also have
\begin{equation}
\label{eq:spec-p-I}
\tCS(H^{\nat p},\ff^p)\cap pI=p(\tCS(H,\ff)\cap I).
\end{equation}
Here $pI=(pa,\,pb)$ for $I=(a,\,b)$.

Next we claim that, when $p$ is sufficiently large,
\begin{equation}
\label{eq:chi-iter}
\chi(H^{\nat p},pI,\ff^p)=\chi(H, I,\ff).
\end{equation}

To see this, denote by $x_i$ the one-periodic orbits of $H$ in the
class $\ff$ and with augmented action in $I$. This is a finite
collection of orbits since $\PP_1(H,[\ff])$ is finite.  Then all
sufficiently large primes $p$ are admissible in the sense of
\cite{GG:gap} for all orbits $x_i$, i.e., 1 has the same multiplicity
as a generalized eigenvalue of the linearized return maps $d\varphi_H$
and $d\varphi_H^p$ at $x_i$ and the two maps have the same number of
eigenvalues in $(-1,\,0)$. (Indeed, it suffices to require $p$ to be
larger than $2$ and larger than the degree of any root of unity among
the eigenvalues of $d\varphi_H$ at $x_i$.) By the Shub--Sullivan
theorem (see \cite{CMPY, SS}), the orbits $x_i$ and $x_i^p$ have the
same Poincar\'e--Hopf index. The orbits $x_i^p$ are the only
$p$-periodic orbits of $H$ in $\ff^p$ with augmented action in $pI$,
and \eqref{eq:chi-iter} follows. Alternatively, one can argue as in
the proof of the case  $\ff\neq 1$ of the theorem; see below.

By \eqref{eq:chi-hom} and since $\chi(H, I,\ff)\neq 0$, we conclude
that
$$
\tHF^{pI}(H^{\nat p},\ff^p)\neq 0.
$$

Now we are in a position to show that $H$ must have at least one
$p'$-periodic orbit in the class $\ff^p$, where $p'$ is the first
prime greater than $p$, provided again that $p$ is sufficiently
large. Then, as the last step of the proof, we will show that this
$p'$-periodic orbit is necessarily simple.

Arguing by contradiction, assume that there are no such orbits. Then
$$
\gap(H^{\nat p'},\ff^p)=\infty,
$$
and, obviously, the augmented action filtration is defined on the
Floer homology for $H^{\nat p'}$ and $\ff^p$. (Of course, the
resulting complex and the homology is zero for any augmented action
interval, but this is not essential at this point.) By roughly
following the line of reasoning from \cite{Gu:nc, Gu:hq} and relying
on the fact that the filtered homology is defined, we will show that
the homology is non-trivial for a certain augmented action interval
and thus arrive at a contradiction with the assumption that $H$ has no
$p'$-periodic orbits in the class $\ff^p$.

Set 
$$
e_+ = \max\left\{ \int_{S^1} \max_M H_t\,dt,\,0\right\} 
$$
and 
$$
e_- =\max\left\{-\int_{S^1} \min_M H_t\,dt,\,0\right\}.
$$
Then 
$$
a_\pm:=(p'-p)e_\pm +|\lambda|n\geq c_\pm,
$$ 
where the constants $c_\pm=c_h(H^s)$ are defined by \eqref{eq:shift}
for the linear homotopies from $H^{\nat p}$ to $H^{\nat p'}$ and from
$H^{\nat p'}$ to $H^{\nat p}$.

Furthermore, recall that $p'-p=o(p)$ as $p\to\infty$; see, e.g.,
\cite{BHP}. Thus, when $p$ is sufficiently large, we have
$$
\gap(H^{\nat p},\ff^p)=p \gap(H,\ff)>a_\pm\geq c_\pm .
$$
Hence, the conditions of Proposition \ref{prop:homotopy} are
satisfied, and the continuation maps
$$
\tHF^{pI}(H^{\nat p},\ff^p)\to \tHF^{pI+a_+}(H^{\nat p'},\ff^p)
$$
and
$$
\tHF^{pI+a_+}(H^{\nat p'},\ff^p)\to \tHF^{pI+a_++a_-}(H^{\nat p},\ff^p)
$$
are defined.

Consider now the following commutative diagram:
$$
\xymatrix{
\tHF^{pI}\big( H^{\nat p}, \ff^p\big)
\ar[d] \ar[rrd]^\cong \\ 
\tHF^{pI+ a_+ }
\big(H^{\nat p'}, 
\ff^p\big)
\ar[rr]& &
\tHF^{pI+ a_+ + a_-}
\big(H^{\nat p}, 
\ff^p\big)
}
$$
Here the diagonal map is an isomorphism. To see this, denote by
$\delta>0$ the distance from the end points of $I$ to
$\tCS(H,\ff)$. Then the distance from the end points of $pI$ to
$\tCS(H^{\nat p},\ff^p)$ is $p\delta$ and, when $p$ is large,
$$
p\delta> a_+ +a_-
$$
again because $p'-p=o(p)$. Hence, the intervals
$(pI+a_+ + a_-)\setminus pI$ and $pI\setminus (pI+a_+ + a_-)$ contain
no points of $\tCS(H^{\nat p},\ff^p)$, and the diagonal map is indeed
an isomorphism. (In fact, this argument shows that, as in the second
part of  Proposition \ref{prop:homotopy}, one can eliminate the shifts
$a_\pm$ and $a_++a_-$ in the continuation maps when $p$ is
sufficiently large.)

Moreover, as we have shown above, $\tHF^{pI}(H^{\nat p},\ff^p)\neq 0$.
Therefore, the middle group $\tHF^{pI+ a_+ }(H^{\nat p'}, \ff^p)$ in
the diagram is also non-trivial, and thus $H$ must have a
$p'$-periodic orbit in the homotopy class $\ff^p$.

It remains to show that this orbit is necessarily simple. However,
otherwise, it would be the $p'$th iteration of a one-periodic orbit in
the homology class $p[\ff]/p'$. This is impossible because $p[\ff]/p'$
is not an integer homology class when $p$, and hence $p'$, are large
since $[\ff]\neq 0$. This completes the proof of the case
$[\ff]\neq 0$ of the theorem.
\end{proof}

\begin{proof}[Proof of Theorem \ref{thm:homology-chi}: the case
  $\ff\neq 1$] The proof follows the same path as in the case where
  $[\ff]\neq 0$, and here we only indicate the necessary changes in
  the argument.

  The key to the proof is the fact that when $\pi_1(M)$ is a torsion
  free hyperbolic group and $\ff\neq 1$ in $\tpi_1(M)$ there exists a
  constant $r(\ff)\in \N$ such that the equation
$$
\ff^p=\fh^q
$$
in $\tpi_1(M)$, where $p$ and $q$ are primes greater than $r(\ff)$ and
$\fh\in\tpi_1(M)$, is satisfied only when $\fh=\ff$ and $p=q$.

To see this, we first note that it is sufficient to prove this fact
for $\pi_1(M)$ rather than $\tpi_1(M)$. In other words, given
$f\in\pi_1(M)$, $f\neq 1$, we need to show that $f^p=h^q$ for
sufficiently large primes $p$ and $q$ (depending on $f$) only when
$h=f$ and $p=q$. To this end, recall that for any hyperbolic group
$G$, every element $f\in G$ of infinite order is contained in a unique
maximal virtually cyclic subgroup $E(f)$ and
$$
E(f)=\{ g\in G\mid g^{-1}f^l g=f^{\pm l}\textrm{ for some } l\in\N\};
$$
see \cite{Ol} (and also \cite{Gr}). Applying this to $f\neq 1$ in
$G=\pi_1(M)$, which is also assumed to be torsion free, and using the
fact that a torsion free and virtually cyclic group is cyclic, we
conclude that $E(f)$ is infinite cyclic, i.e., $\Z$. Furthermore,
$h\in E(f)$ as a consequence of the condition $f^p=h^q$. Indeed,
$$
h^{-1}f^p h=h^{-1}h^q h= h^q=f^p.
$$
This reduces the question to the case where $f$ and $h$ belong to the
infinite cyclic group $E(f)$, and in this case the result is
obvious.\footnote{The authors are grateful to Denis Osin for this
  argument.}

From now on, we require that $p\geq r(\ff)$. (Later we will need to
introduce additional lower bounds for $p$.) As in the proof of the
first case of the theorem, assume that $H$ has no simple $p$-periodic
orbits in the class $\ff^p$. Then every $p$-periodic orbit is the
$p$th iteration of a one-periodic orbit and, by the above observation
(with $p=q$), this one-periodic orbit must also be in the class
$\ff$. Clearly, \eqref{eq:spec-p}, \eqref{eq:gap-p}, and
\eqref{eq:spec-p-I} still hold.

Furthermore, \eqref{eq:chi-iter} also holds, i.e.,
$\chi(H^{\nat p},pI,\ff^p)=\chi(H,I,\ff)$, although now the reason is
slightly different. Consider the set $F$ of the initial conditions
$x(0)$ for all one-periodic orbits of $H$ in the class $\ff$ and with
augmented action in $I$. Since the end points of $I$ are outside
$\tCS(H,\ff)$, the set $F$ is closed. Under a small non-degenerate
perturbation $\tH$ of $H$, the set $F$ splits into a finite collection
of the initial conditions of the orbits $\tilde{x}_i$ of $\tH$ in
$\ff$ with augmented action in $I$, and $\chi(H,I,\ff)$ is the sum of
the Poincar\'e--Hopf indices of the orbits $\tilde{x}_i$. We can
furthermore ensure that there are no roots of unity among the Floquet
multipliers of these orbits.  By our assumptions, $F$ is also the set
of the initial conditions for all $p$-periodic orbits of $H$ in the
class $\ff^p$ with augmented action in $pI$. If $\tH$ is sufficiently
close to $H$, the only $p$-periodic orbits of $\tH$ in $\ff^p$ with
augmented action in $pI$ are $\tilde{x}_i^p$. Hence,
$\chi(H^{\nat p},pI,\ff^p)$ is the sum of the Poincar\'e--Hopf indices
of the orbits $\tilde{x}_i^p$.  When $p>2$, the orbits $\tilde{x}_i$
and $\tilde{x}_i^p$ have the same Poincar\'e--Hopf index due to the
assumption that none of the Floquet multipliers is a root of unity.
As a consequence, we have \eqref{eq:chi-iter}.

The rest of the proof is identical to the argument in the case where
$[\ff]\neq 0$ except for the very last step.  Thus we have proved the
existence of a $p'$-periodic orbit $x$ in the class $\ff^p$ and now
need to show that this orbit is simple. Assume the contrary. Then,
since $p'$ is prime, $x$ is necessarily the $p'$th iteration of a
one-periodic orbit in some class $\fh$. We have
$$
\ff^p=\fh^{p'},
$$
and hence $p'=p$ when $p>r(\ff)$. This is impossible since $p'$ is the
first prime greater than $p$.
\end{proof}

Turning to Theorem \ref{thm:atoro}, note that, as in \cite{Gu:nc}, a
more general result holds. Namely, recall that the local Floer
homology $\HF(x)$ is associated to an isolated periodic orbit $x$ of
$H$. The group $\HF(x)$, already mentioned in Section
\ref{sec:filtration}, is roughly speaking the homology of the Floer
complex generated by the orbits which $x$ splits into under a
non-degenerate perturbation; see, e.g., \cite{GG:gap} for more
details. In particular, when $x$ is non-degenerate, $\HF(x)$ is equal
to the ground ring and concentrated in degree $\MUCZ(x)$. We have the
following generalization of Theorem~\ref{thm:atoro}:

\begin{Theorem}
\label{thm:atoro2}
Assume that the class $[\omega]$ is atoroidal and let $H$ be a
Hamiltonian having an isolated one-periodic orbit $x$ with homotopy
class $\ff$ such that $\HF(x)\neq 0$ and that $[\ff] \neq 0$ in
$\H_1(M;\Z)/\Tor$ and $\Pp_1(H,[\ff])$ is finite. Then, for every
sufficiently large prime $p$, the Hamiltonian $H$ has a simple
periodic orbit in the homotopy class $\ff^p$ and with period either
$p$ or $p'$, where $p'$ is the first prime greater than $p$. Moreover,
when $\pi_1(M)$ is hyperbolic and torsion free, the condition
$[\ff]\neq 0$ can be replaced by $\ff\neq 1$.
\end{Theorem}

The main new point here is the ``moreover'' part of the theorem.  The
case of the theorem where $[\ff]\neq 0$, proved in \cite{Gu:nc}, is
included for the sake of completeness. The proof of the ``moreover''
part is a combination of the proof of \cite[Thm.\ 3.1]{Gu:nc} and the
proof of the case $\ff\neq 1$ of Theorem \ref{thm:homology-chi}. Here
we only briefly touch upon this argument.

\begin{proof}[On the proof of Theorem \ref{thm:atoro2}]
  The key difference between the settings of Theorems \ref{thm:atoro}
  and \ref{thm:atoro2} and that of Theorem \ref{thm:homology-chi} is
  that now the class $[\omega]$ is atoroidal and thus we have the
  standard action filtration $\HF^I(H;\ff)$ on the Floer homology of
  $H$ rather than the augmented action filtration. On the level of
  complexes, the action filtration is strictly monotone, i.e., the
  differential is strictly action decreasing. (In contrast, the
  augmented action is only non-increasing; see the discussion in
  Section \ref{sec:filtration}.)  As a consequence,
  $\HF^I(H;\ff)\neq 0$ when $I$ is a small interval centered at the
  action $\CA_H(x)$ whenever $x$ and $\CA_H(x)$ are isolated and
  $\HF(x)\neq 0$. Furthermore, when $H$ has no simple $p$-periodic
  orbits in the class $\ff^p$, we have
  $\HF^{pI}(H^{\nat p};\ff^p)\neq 0$ by the persistence of the local
  Floer homology results from \cite{GG:gaps}. With this in mind, one
  argues essentially word-for-word as in the proof of Theorem
  \ref{thm:homology-chi}, with some straightforward
  simplifications. We omit the details. \end{proof}

\subsection{Proof of Theorem \ref {thm:homology-hyperbolic}}
Arguing by contradiction, assume that $H$ has only finitely many
simple periodic orbits with homotopy class in
$\ff^\N=\{\ff^k\mid k\in\N\}$, where $\ff=\llbracket x \rrbracket$. We
denote these orbits by $x_j$ and let $k_j$ be the period of $x_j$. Set
$F=H^{\nat 2k_0}$ and $y_j=x_j^{2k_0/k_j}$, where $k_0$ is the least
common multiple of the periods $k_j$. The orbits $y_j$ are the
one-periodic orbits of $F$.

We claim that for all $k\in\N$ every $k$-periodic orbit $z$ of $F$ in
the collection of the homotopy classes $\ff^\N$ is the $k$th iteration
of one of the orbits $y_j$. Indeed, then $z$ is also a
$2kk_0$-periodic orbit of $H$ in $\ff^\N$. Hence,
$z=x_j^{2kk_0/k_j}=y_j^k$ for some $j$.

Thus we have a collection of free homotopy classes $\ff^{\N}$
generated by $\ff\in \tpi_1(M)$, and a Hamiltonian $F$ with finitely
many one-periodic orbits $y_j$ in $\ff^\N$ and no other simple
periodic orbits in $\ff^\N$. One of these orbits, $h$, is an even
iteration of the original hyperbolic orbit. Hence, $h$ is hyperbolic
with an even number of eigenvalues in $(-1,\,0)$.

From now on we focus on the Hamiltonian $F$ and its periodic orbits.
Set $\fh=\llbracket h \rrbracket$; clearly, $\fh^\N\subset \ff^\N$.

Let us fix reference curves and trivializations for the collection
$\fh^\N$. Namely, it is convenient to take $h$ as a reference curve
for $\fh$. Then the reference trivialization is fixed by the condition
that $\Delta_F(\bh)=0$, where $\bh$ stands for $h$ equipped with
the ``identity'' capping. (Such a trivialization exists since $h$ is
hyperbolic and has an even number of eigenvalues in $(-1,\,0)$, and
hence the mean index of $h$ with respect to any trivialization is an
even integer.)  The class $\fh^k$ is then given the iterated reference
curve $h^k$ and the ``iterated'' trivialization.  We fix cappings of
the orbits $y_j$, suppressed in the notation, and equip the iterated
orbits with ``iterated'' cappings.  As a consequence, the action, the
mean index, and the augmented action are homogeneous under iterations
for periodic orbits in $\fh^\N$.  (It is essential here that all free
homotopy classes in $\fh^\N$ are distinct and nontrivial, and hence
the reference curve and the trivialization are well defined.)

Without loss of generality, by adding if necessary a constant to $F$,
we can ensure that $\tA_F(h)=0$.

By our assumptions, we have 
$$
\tCS(F^{\nat k},\fh^k)=k\tCS(F,\fh)
$$
and 
$$
\gap(F^{\nat k},\fh^k)=k\gap(F,\fh).
$$
It follows that the augmented action filtered Floer homology of
$F^{\nat k}$ is defined when $k$ is large enough.

Furthermore, let $I$ be an interval such that $0=\tA_F(h)$ is the only
point in $\tCS(F,\fh)\cap I$, and the end points of $I$ are not in
$\tCS(F,\fh)$. Then we also have
\begin{equation}
\label{eq:kI}
\tCS(F^{\nat k},\fh^k)\cap kI=\{0\}
\end{equation}
and the end points of $kI$ are outside $\tCS(F^{\nat k},\fh^k)$.

\begin{Lemma}
\label{lemma:hom}
$\tHF^{k_i I}(F^{\nat k_i}, \fh^{k_i})\neq 0$ for some sequence
$k_i\to\infty$.
\end{Lemma}

Assuming the lemma, let us finish the proof of the theorem.  Similarly
to the proof of Theorem \ref{thm:homology-chi}, set
$$
a_+ = \max\left\{ \int_{S^1} \max_M F_t\,dt,\,0\right\} +|\lambda|n
$$
and 
$$
a_- =\max\left\{-\int_{S^1} \min_M F_t\,dt,\,0\right\}+ |\lambda|n .
$$
These constants are greater than or equal to the constants $c_h$ given
by \eqref{eq:shift} for the linear homotopies from $F^{\nat k}$ to
$F^{\nat (k+1)}$ and from $F^{\nat (k+1)}$ to $F^{\nat k}$, and
denoted again by $c_\pm$. Thus, when $k$ is sufficiently large,
$$
\gap(F^{\nat k},\fh^k)=k \gap(F,\fh)>a_\pm\geq c_\pm .
$$
Hence, by Proposition \ref{prop:homotopy}, the continuation maps
$$
\tHF^{kI}(F^{\nat k},\fh^k)\to \tHF^{kI+a_+}(F^{\nat (k+1)},\fh^k)
$$
and
$$
\tHF^{kI+a_+}(F^{\nat (k+1)},\fh^k)\to \tHF^{kI+a_++a_-}(F^{\nat k},\fh^k)
$$
are defined. 

Let $\delta>0$ be the distance from the end points of $I$ to
$\tCS(F,\fh)$. Then the distance from the end points of $kI$ to
$\tCS(H^{\nat k},\fh^k)$ is $k\delta$. When $k$ is large,
$k\delta>a_++a_-$, and the intervals $(kI+a_+ + a_-)\setminus kI$ and
$kI\setminus (kI+a_+ + a_-)$ contain no points of
$\tCS(F^{\nat k},\fh^k)$. Thus the natural quotient-inclusion map
$$
\tHF^{kI}(F^{\nat k},\fh^k)\to \tHF^{kI+a_++a_-}(F^{\nat k},\fh^k)
$$
is an isomorphism.

Let now $k$ be one of the sufficiently large entries in the sequence
$k_i$ from Lemma \ref{lemma:hom}.  Consider the following commutative
diagram:
$$
\xymatrix{
\tHF^{k I}\big( F^{\nat k}, \fh^{k}\big)
\ar[d] \ar[rrd]^\cong \\ 
\tHF^{k I+ a_+ }
\big(F^{\nat (k+1)}, 
\fh^{k}\big)
\ar[rr]& &
\tHF^{k I+ a_+ + a_-}
\big(F^{\nat k}, 
\fh^{k}\big) 
}
$$
Here the diagonal map is an isomorphism and, by Lemma \ref{lemma:hom},
$\tHF^{kI}(F^{\nat k},\fh^{k})\neq 0$. Therefore, the middle group
$\tHF^{kI+ a_+ }(F^{\nat (k+1)}, \fh^{k})$ in the diagram is also
non-trivial, and $F$ has a $(k+1)$-periodic orbit $z$ in the homotopy
class $\fh^{k}$.

We have $z=y_j^{{k}+1}$ for some $j$. Furthermore, $\fh=\ff^a$ and
$\llbracket y_j \rrbracket=\ff^b$ for some $a$ and $b$ in $\N$. Thus
$\ff^{ak}=\ff^{b(k+1)}$. Since all homotopy classes in $\ff^\N$ are
distinct, we infer that $ak=b(k+1)$, where $a$ is independent of
$k$. This is clearly impossible for $k\geq a$ since $k$ and $k+1$ are
relatively prime, and we have arrived at a contradiction.  To finish
the proof of the theorem, it remains to prove the lemma.

\begin{proof}[Proof of Lemma \ref{lemma:hom}] Throughout the proof, it
  is convenient to interpret the iterated Hamiltonian $F^{\nat k}$ as
  the $k$-periodic Hamiltonian $F_t$ with $t\in\R/k\Z$.  Furthermore,
  without loss of generality, we can ensure that $\lambda/2=1$ by
  rescaling $\omega$, and thus $\tA_F=\CA_F-\Delta_F$.  Since
  $\tA_F(h)=0$ and $\Delta_F(\bh)=0$, we also have
\begin{equation}
\label{eq:h-index}
\CA_F(\bh)=0=\Delta_F(\bh).
\end{equation}
Recall that, by our assumptions, $0$ is the only point of the action
spectrum $\tCS(F^{\nat k},\fh^k)$ in the interval $kI$ (see
\eqref{eq:kI}) and that $y_j^k$ are the only $k$-periodic orbits of
$F$.

Fix a neighborhood $U$ of $h$ which does not intersect any of the
other orbits $y_j$ and a small parameter $\eps>0$, depending on $U$,
to be specified later. There exists a sequence $k_i\to\infty$ such
that for all $j$ and $k_i$, we have
\begin{equation}
\label{eq:eps}
\big\|\Delta_{F^{\nat k_i}}(y_j^{k_i})\big\|_{2\NT}<\eps,
\end{equation}
where $\|a\|_{2\NT}$ stands for the distance from
$a\in S^1_{2\NT}=\R/2\NT\Z$ to $0$ or, equivalently, from $a\in\R$ to
the nearest point in the lattice $2\NT\Z\subset \R$. Here we can treat
the mean index
$$
\Delta_{F^{\nat k_i}}(y_j^{k_i})=k_i\Delta_F(y_j)
$$ 
as a real number when $y_j$ is capped or as a point in $S^1_{2\NT}$
when the capping is discarded.

To prove \eqref{eq:eps}, consider the torus $\T^m=(S^1_{2\NT})^m$
where $m$ is the number of the orbits $y_j$ and set
$$
\Delta=\big(\Delta_F(y_1),\ldots,\Delta_F(y_m)\big)\in\T^m .
$$
The closure $\Gamma$ of the orbit $\{k\Delta\mid k\in\Z\}$ is a
subgroup in $\T^m$ and, for every $k_0$, the set
$\{k\Delta\mid k>k_0\}$ is dense in $\Gamma$. Hence, the point
$k\Delta$ is within the $\eps$-neighborhood of $0\in\Gamma$ for
infinitely many values of $k$.

From now on $k$ will stand for one of the entries in the sequence
$k_i$.

Let $G$ be a $C^\infty$-small, non-degenerate perturbation of
$F^{\nat k}$ equal to $F^{\nat k}$ on the neighborhood $U$. The orbits
$y_j^k$, other than $h^k$, split into a finite collection of
non-degenerate orbits of $G$ in the class $\fh^k$. Among these we are
interested exclusively in the orbits with augmented action in
$kI$. These orbits can only come from the orbits $y_j^k$ with action
in $kI$, i.e., by \eqref{eq:kI}, from the orbits $y_j$ with
$\tA_F(y_j)=0$. We denote the resulting orbits of $G$ by $z_j$. (The
number of the orbits $z_j$ may be different from the number of the
orbits $y_j$.) It is clear that the orbits $z_j$ do not enter $U$ and
that
\begin{equation}
\label{eq:eta}
\big|\tA_G(z_j)\big|\leq \eta,
\end{equation}
where $\eta=O\big(\|F^{\nat k}-G\|_{C^1}\big)$.

It suffices now to show that when $\eps>0$ is sufficiently small the
orbit $h^k$ of $G$ is closed (i.e., a cycle), but not exact, in the
Floer complex $\tCF^{kI}(G,\fh^k)$. To this end, we will prove that
$h^k$ cannot be connected to any of the orbits $z_j$ by a Floer
trajectory of relative index $\pm 1$.

By \eqref{eq:h-index}, we have
$$
\Delta_{G}(\bh^k)=k\Delta_F(\bh)=0 \textrm{ and }
\CA_{G}(\bh^k)=k\CA_F(\bh)=0
$$
In particular, $\MUCZ(h^k)=\Delta_G(h^k)=0$.

Let now $\bz$ be one of the capped orbits $\bz_j$. Our goal is to show
that every Floer trajectory $u$ connecting the capped orbits $\bz$ and
$\bh^k$ has relative index different from $\pm 1$.  Since
$\MUCZ(\bh^k)=0$ and by \eqref{eq:mean-CZ}, it is enough to prove that
\begin{equation}
\label{eq:Delta}
|\Delta_G(\bz)|> n+1.
\end{equation}

The orbit $z$ does not enter $U$. Thus, by \cite[Thm.\
3.1]{GG:hyperbolic}, there exists a constant $e>0$, depending on $U$,
but not on $k$, such that the energy of $u$ is bounded from below by
$e$. In other words, using the fact that $\CA_G(\bh^k)=0$, we have
$$
|\CA_G(\bz)|>e.
$$
Set $\eps<e$.

By \eqref{eq:eps}, $\Delta_G(\bz)\in (\ell-\eps,\,\ell+\eps)$ for some
$\ell\in\Z$. If $\ell=0$, and hence $|\Delta_G(\bz)|<\eps$, we also
have
$$
|\CA_G(\bz)|<\eps+\eta
$$
by \eqref{eq:eta}. This is impossible when $\eta>0$ is smaller than
$|e-\eps|$, i.e., when $G$ is sufficiently $C^1$-close to
$F^{\nat k}$, since $\eps<e$. Thus $\ell\neq 0$, and therefore
$$
|\Delta_G(\bz)|> 2\NT-\eps.
$$
Recall that $\NT\geq n/2+1$ by the assumptions of the theorem. Hence,
when $\eps<1$, we have
$$
|\Delta_G(\bz)|\geq n+2-\eps>n+1.
$$ 
This proves \eqref{eq:Delta}, completing the proof of the lemma and of
the theorem.
\end{proof}

\subsection{Proof of Theorem \ref{thm:generic}} The argument is quite
standard; it follows the same line of reasoning as the proof of
\cite[Prop.\ 1.6]{GG:generic}, which, in turn, has a lot of
similarities with, e.g., an argument in \cite{Hi}.

\begin{proof}
  
  It suffices to show that a Hamiltonian diffeomorphism
  $\varphi_H\in\FF_\ff$ has a non-hyperbolic periodic orbit in some
  homotopy class $\ff^k$ or $\varphi_H$ has infinitely many periodic
  orbits in $\ff^\N$. Indeed, the presence of a non-hyperbolic
  periodic orbit $x$ implies, by the Birkhoff--Lewis--Moser
  fixed-point theorem (see \cite{Mo}), the existence of infinitely
  many periodic orbits in a tubular neighborhood of $x$
  $C^\infty$-generically for Hamiltonian diffeomorphisms close to
  $\varphi_H$.
 
  To this end, observe that since $\HF(H,\ff)=0$ due to the condition
  $\ff\neq 1$, the Hamiltonian $H$ necessarily has one-periodic orbits
  in the class $\ff$ with odd and with even Conley--Zehnder
  indices. As a consequence, it has either a non-hyperbolic
  one-periodic orbit or a hyperbolic orbit with an odd number of real
  Floquet multipliers in the interval $(-1,\,0)$. In the former case
  the proof is finished.

  In the latter case, let us apply this argument to $\varphi_H^2$. By
  the assumptions of the theorem $\ff^2\neq 1$ and thus
  $\HF(H^{\nat 2},\ff^2)=0$.  Hence, $H$ has two-periodic orbits in
  the class $\ff^2$ with odd and with even Conley--Zehnder
  indices. The second iterations of one-periodic orbits from the class
  $\ff$ are necessarily positive hyperbolic. Therefore, there exists a
  simple two-periodic orbit in $\ff^2$, which is either non-hyperbolic
  or hyperbolic with an odd number of real Floquet multipliers in the
  interval $(-1,\,0)$. In the former case, the proof is finished, and
  in the latter we repeat this process for $\varphi^4$ and so forth.

  As a result, we will either find a non-hyperbolic orbit in some
  class $\ff^k$ or construct a sequence of simple periodic orbits in
  $\ff^\N$.
\end{proof}

\begin{Remark}
  As is clear from the proof, it is sufficient to assume only that
  $\ff^k\neq 1$ when $k$ is a power of $2$. In this case, the result
  asserts the generic existence of infinitely many periodic orbits in
  the set of homotopy classes $\ff^\N$, although these orbits may now
  be contractible.
\end{Remark}

\end{document}